\documentclass[12pt]{article}
\usepackage{amsmath}
\usepackage{amssymb}
\usepackage{amsfonts}
\usepackage{eucal}
\usepackage{amsthm}
\theoremstyle{plain}
\newtheorem{lem}{Lemma}
\newtheorem{thm}{Theorem}
\newtheorem*{thm*}{Theorem}

\newtheorem{prop}{Proposition}
\newtheorem{cor}{Corollary}
\theoremstyle{definition}
\newtheorem{rem}{Remark}
\newtheorem*{proof1}{Proof of Theorem \ref{ch2}}
\newtheorem*{proof2}{Proof of Theorem \ref{qpnn1}}
\newtheorem*{proof3}{Proof of Theorem \ref{eta2pq}}
\newtheorem*{proof4}{Proof of Theorem \ref{off}}
\newtheorem*{proof5}{Proof of Theorem \ref{eta111}}
\newtheorem*{proof6}{Proof of Theorem \ref{spiess}}
\def\R{\mathbb R}
\def\Q{\mathbb Q}
\def\P{\mathcal P}

\def\F{\mathcal F}
\def\Sym{\operatorname{Sym}}
\def\QSy{\operatorname{QSym}}
\def\zt{\zeta}
\def\la{\lambda}

\def\Si{\Sigma}

\begin{document}
\title{Harmonic-Number Summation Identities, Symmetric Functions,
and Multiple Zeta Values}
\author{Michael E. Hoffman \\
\small Dept. of Mathematics \\[-0.8ex]
\small U. S. Naval Academy, Annapolis, MD 21402 USA\\[-0.8ex]
\small \texttt{meh@usna.edu}}
\date{November 24, 2015 
}
\maketitle
\begin{abstract}
We show how infinite series of a certain type involving generalized 
harmonic numbers can be computed using a knowledge of symmetric 
functions and multiple zeta values.
In particular, we prove and generalize some identities recently 
conjectured by J. Choi, and give several more families of identities
of a similar nature.
\end{abstract}
\section{Introduction}
Let $H_n^{(r)}$ denote the generalized harmonic number $\sum_{j=1}^n\frac1{n^r}$; 
if $r=1$ we omit the superscript.
This paper is concerned with series of the form
\begin{equation}
\label{tsum}
\sum_{n=1}^\infty \frac{F(H_n,H_n^{(2)},\dots,H_n^{(j)})}
{n^{s_1}(n+1)^{s_2}\cdots (n+k-1)^{s_k}} ,
\end{equation}
where $F(x_1,\dots,x_j)\in\Q[x_1,\dots,x_j]$ and $s_1,\dots,s_k$ are 
nonnegative integers with $s_1+\dots+s_k\ge 2$.
There are many interesting identities giving closed forms for such
sums, starting with the formulas
\begin{equation}
\label{eulers}
\sum_{n=1}^\infty \frac{H_n}{n^2}=2\zt(3)\quad\text{and}\quad
\sum_{n=1}^\infty \frac{H_n}{n^3}=\frac54\zt(4),
\end{equation}
both due to Euler \cite{E}.  Many similar formulas have been established
since, and there is an extensive literature; see, e.g., 
\cite{BB, ChCh, Ch, Con, FS, Me, S, So, Zh}.
\par
In the 1990's, multiple zeta values were introduced by the author \cite{H1}
and D. Zagier \cite{Z}.  These are defined by
\begin{equation}
\label{mzv}
\zt(i_1,i_2,\dots,i_k)=\sum_{n_1>n_2>\dots>n_k\ge 1}\frac1{n_1^{i_1}n_2^{i_2}\cdots 
n_k^{i_k}}
\end{equation}
for positive integers $i_1,i_2,\dots,i_k$ with $i_1>1$.  
For integer $s\ge 2$, any sum of the form
\[
\sum_{n=1}^\infty \frac{H_n^{(r)}}{n^s},
\]
which includes Euler's examples (\ref{eulers}), is readily expressible 
in terms of multiple zeta values as $\zt(r+s)+\zt(s,r)$.  
In recent decades intensive study has led to the development of an 
extensive theory of multiple zeta values, which we summarize in \S3 below.
The point of this paper is that sums of form (\ref{tsum}) can be
expressed in terms of multiple zeta values, and using some facts
about symmetric functions and multiple zeta values often allows such
expressions to be put in a particularly simple form.
\par
Recently J. Choi \cite[Corollary 3]{C} proved a sequence of identities:
\begin{align}
\label{cho11}
\sum_{n=1}^\infty \frac{H_n}{(n+1)(n+2)}&=1\\
\label{cho12}
\frac12\sum_{n=1}^\infty \frac{H_n^2-H_n^{(2)}}{(n+1)(n+2)}&=1\\
\label{cho13}
\frac16\sum_{n=1}^\infty \frac{H_n^3-3H_nH_n^{(2)}+2H_n^{(3)}}{(n+1)(n+2)}&=1\\
\label{cho14}
\frac1{24}\sum_{n=1}^\infty \frac{H_n^4-6H_n^2H_n^{(2)}+8H_nH_n^{(3)}+3(H_n^{(2)})^2-6H_n^{(4)}}
{(n+1)(n+2)}&=1.
\end{align}
The sequence $P_k$ of multivariate polynomials in the numerators, which
starts
\[
P_1(x_1)=x_1,\ P_2(x_1,x_2)=\frac12(x_1^2-x_2),\ P_3(x_1,x_2,x_3)=
\frac16(x_1^3-3x_1x_2+2x_3),\dots
\]
turns out to be well-known in the theory of symmetric functions; in fact
\[
P_k(p_1,p_2,\dots,p_k)=e_k ,
\]
where $p_i$ is the $i$th power sum and $e_i$ is the $i$th elementary
symmetric function.
We discuss the $P_k$ in \S2 below.  
We also discuss another sequence of polynomials $Q_k$, which are 
simply the $P_k$ without signs, i.e.,
\[
Q_1(x_1)=x_1,\ Q_2(x_1,x_2)=\frac12(x_1^2+x_2),\ Q_3(x_1,x_2,x_3)=
\frac16(x_1^3+3x_1x_2+2x_3),\dots
\]
(The $Q_k$ express the complete symmetric functions in terms of
power sums.)
Not only is the identity 
\begin{equation}
\label{ch1}
\sum_{n=1}^\infty \frac{P_k(H_n,H_n^{(2)},\dots,H_n^{(k)})}{(n+1)(n+2)}=1
\end{equation}
true for all $k$, but in fact it has a generalization involving the
$Q_k$, of which Choi proved \cite[Corollary 5]{C} the first few cases:
\begin{align}
\label{cho21}
\sum_{n=1}^\infty \frac{H_n^2}{(n+1)(n+2)}&=1+\zt(2)\\
\label{cho22}
\frac12\sum_{n=1}^\infty \frac{H_n(H_n^2-H_n^{(2)})}{(n+1)(n+2)}&=
1+\zt(2)+\zt(3)\\
\label{cho23}
\frac16\sum_{n=1}^\infty \frac{H_n(H_n^3-3H_nH_n^{(2)}+2H_n^{(3)})}{(n+1)(n+2)}&=
1+\zt(2)+\zt(3)+\zt(4)
\end{align}
(in these identities there is an erroneous factor of 2 in \cite{C} 
which we have removed).
The general result as follows. (We make the convention
that $H_0^{(r)}=0$ for all $r$, so that the result holds if $k=l=0$.)
\begin{thm}
\label{ch2}
If $P_k,Q_k$ are the polynomials discussed above, $k,l$ nonnegative
integers, then
\begin{multline*}
\sum_{n=0}^\infty
\frac{Q_l(H_n,H_n^{(2)},\dots,H_n^{(l)})P_k(H_n,H_n^{(2)},\dots,H_n^{(k)})}
{(n+1)(n+2)}\\
=\begin{cases}
\sum_{j=0}^k\binom{k+l-j}{k+1-j}\zt(k+l-j)-\zt(l),& l\ge 2,\\
\sum_{j=0}^{k-1}\zt(k+1-j)+1,& l=1,\\
1,& l=0.
\end{cases}
\end{multline*}
\end{thm}
\par
The denominator $(n+1)(n+2)$ appearing in Theorem \ref{ch2} can be replaced 
with other polynomials.  The analogue of
Theorem \ref{ch2} for the denominator $n(n+1)$ is especially simple.
\begin{thm}
\label{qpnn1}
For $k,l$ nonnegative integers with $k+l\ge 1$,
\[
\sum_{n=1}^\infty
\frac{Q_l(H_n,H_n^{(2)},\dots,H_n^{(l)})P_k(H_n,H_n^{(2)},\dots,H_n^{(k)})}
{n(n+1)}=\binom{k+l+1}{k+1}\zt(k+l+1).
\]
\end{thm}
When the denominator is $n^2$ we have the following result.
\begin{thm}
\label{eta2pq}
For the polynomials $P_k,Q_k$ discussed above, $k$ nonnegative,
\begin{align}
\label{qn2}
\sum_{n=1}^\infty
\frac{Q_k(H_n,H_n^{(2)},\dots,H_n^{(k)})}{n^2}&=(k+1)\zt(k+2)\\
\label{pn2}
\sum_{n=1}^\infty
\frac{P_k(H_n,H_n^{(2)},\dots,H_n^{(k)})}{n^2}&=\frac{k+3}{2}\zt(k+2)
+\frac12\sum_{j=2}^k\zt(j)\zt(k+2-j) .
\end{align}
\end{thm}
\begin{rem}
For $k=1$, both equations give the first of Euler's formulas
(\ref{eulers}) mentioned above.
Equation (\ref{qn2}) can be deduced from \cite[Corollary 1]{CoCa};
the special cases $k=2$ and $k=3$ appear as \cite[eqn. (1.5a)]{Ch} and
\cite[eqn. (2.3f)]{Zh} respectively, and $k=4$ appears in \cite{Con}.  
The sum and difference of
equations (\ref{qn2}) and (\ref{pn2}) for $k=3$ can be recognized as 
\cite[eqn. (2.5e)]{Zh} and \cite[eqn. (2.5d)]{Zh} respectively.
\end{rem}
In fact we have a general formula for 
\[
\sum_{n=1}^\infty
\frac{Q_k(H_n,H_n^{(2)},\dots,H_n^{(k)})P_l(H_n,H_n^{(2)},\dots,H_n^{(l)})}
{n^2}
\]
in terms of multiple zeta values (Theorem \ref{eta2eh} below), but it 
can only be reduced to ordinary zeta values in certain cases.
\par
We are able to obtain results for other denominators as well,
including the following.
\begin{thm}
\label{off}
For nonnegative integers $k,l$,
\[
\sum_{n=0}^\infty\frac{Q_l(H_{n+1},H_{n+1}^{(2)},\dots,H_{n+1}^{(l)})
P_k(H_n,H_n^{(2)},\dots,H_n^{(k)})}{(n+1)^2}=\binom{l+k+1}{l}\zt(l+k+2) .
\]
\end{thm}
\begin{rem}
Several special cases of Theorem \ref{off} occur in the literature.
The case $k=l=1$ is
\[
\sum_{n=1}^\infty \frac{H_nH_{n+1}}{(n+1)^2}=3\zt(4),
\]
which appears as \cite[eqn. (1.2a)]{Ch}.  The cases $(k,l)=(2,1)$
and $(k,l)=(1,2)$ appear in \cite{Zh} as equations (2.3c) and
(2.3e) respectively. 
\end{rem}
\begin{thm}
\label{eta111}
For nonnegative integers $k,l$ with $k+l\ge 1$,
\begin{multline*}
\sum_{n=1}^\infty\frac{Q_l(H_n,H_n^{(2)},\dots,H_n^{(l)})P_k(H_n,H_n^{(2)},\dots,
H_n^{(k)})}{n(n+1)(n+2)}\\
=\begin{cases}\frac12\left[\binom{k+l+1}{k}\zt(k+l+1)-
\sum_{j=0}^k\binom{k+l-j}{l+1-j}\zt(k+l-j)-\zt(l)\right],&l\ge 2,\\
\frac12\left[(k+2)\zt(k+2)-\sum_{j=0}^{k-1}\zt(k+1-j)-1\right],&l=1,\\
\frac12(\zt(k+1)-1),& l=0.\end{cases}
\end{multline*}
\end{thm}
\par
Equation (\ref{ch1}) can be generalized in another direction.
\begin{thm} For integers $k\ge 0$ and $q\ge 2$,
\label{spiess}
\[
\sum_{n=0}^\infty\frac{P_k(H_n,H_n^{(2)},\dots,H_n^{(k)})}{(n+1)(n+2)\cdots (n+q)}
=\frac1{(q-1)!}\cdot\frac1{(q-1)^{k+1}} .
\]
\end{thm}
This actually follows from a result of J. Spie\ss\ \cite{S}, but we prove 
it by our own methods.  As a corollary we get the formula
\begin{equation}
\label{tail}
\sum_{n=1}^\infty\frac{P_k(H_n,H_n^{(2)},\dots,H_n^{(k)})}{n(n+1)\cdots (n+q-1)}
=\frac1{(q-1)!}\left[\zt(k+1)-\sum_{j=1}^{q-2}\frac1{j^{k+1}}\right]  
\end{equation}
for integers $k>0$ and $q\ge 2$, which generalizes the case $l=0$ of
Theorems \ref{qpnn1} and \ref{eta111}.
We note that the case $k=1$ of identity (\ref{tail}) coincides with
Theorem 1 of \cite{BC}.
\par
Our main technical tool is the introduction of a
class of functions $\eta_{s_1,\dots,s_k}$, which we call $H$-functions, 
from the quasi-symmetric functions (a superalgebra of the symmetric
functions) to the reals such that
\[
\eta_{s_1,s_2,\dots,s_k}(p_1)=\sum_{n=1}^\infty\frac{H_n}{n^{s_1}(n+1)^{s_2}
\cdots (n+k-1)^{s_k}} .
\]
Here $(s_1,s_2,\dots,s_k)$ is a sequence of nonnegative integers whose
sum is 2 or more.
We are able to express $\eta_{s_1,\dots,s_k}(u)$, for any quasi-symmetric
function $u$, in terms of multiple zeta values.
Furthermore, for three particular choices of the sequence $(s_1,\dots,s_k)$,
namely $(2)$, $(1,1)$ and $(0,1,1)$, we are able to write simple formulas 
for $\eta_{s_1,\dots,s_k}(u)$ when $u$ is a product of elementary and complete 
symmetric functions.
The proofs rely on a certain class of symmetric functions that came
up in earlier work of the author \cite{H4}, along with some results
about multiple zeta values.
By taking linear combinations of $\eta_2$, $\eta_{1,1}$, and $\eta_{0,1,1}$,
we are able to prove many more summation formulas, as discussed in \S5.
\section{Symmetric and quasi-symmetric functions}
Let $x_1,x_2,\dots$ be a countable set of indeterminates, each of 
which has degree 1.
Let $\P$ be the set of formal power series in the $x_i$ of bounded degree.
A symmetric ``function'' is an element of $u\in\P$ such that the
coefficient of any monomial $x_{i_1}^{a_1}\cdots x_{i_k}^{a_k}$ 
(with the $i_j$ distinct) in $u$ is the same as the coefficient of the 
monomial $x_1^{a_1}\cdots x_k^{a_k}$ in $u$.  
The symmetric functions form a ring $\Sym$.
For a partition $\la=(\la_1,\la_2,\dots,\la_k)$ of $n$, the monomial
symmetric function $m_{\la}$ is the ``smallest'' symmetric function
that contains the monomial $x_1^{\la_1}x_2^{\la_2}\cdots x_k^{\la_k}$.
A symmetric function of degree $n$ is a linear combination of the monomial 
symmetric functions $m_{\la}$ with $\la$ running over partitions of $n$.
Some important symmetric functions are the power sums
\[
p_k=m_{(k)}=x_1^k+x_2^k+\cdots ,
\]
the elementary symmetric functions
\[
e_k=m_{(\underbrace{\scriptstyle{1,\dots,1}}_k)}=x_1x_2\cdots x_k+x_2x_3\cdots x_{k+1}+\cdots ,
\]
and the complete symmetric functions
\[
h_k=\sum_{|\la|=n}m_{\la} .
\]
\par
An element $u\in\P$ such the the coefficient in $u$ of any monomial
$x_{i_1}^{a_1}\cdots x_{i_k}^{a_k}$ with $i_1<i_2<\dots<i_k$ is the same
as that of $x_1^{a_1}x_2^{a_2}\cdots x_k^{a_k}$ is called quasi-symmetric.
This is a weaker condition than being symmetric:  every symmetric 
function is quasi-symmetric, but there are quasi-symmetric functions like
\begin{equation}
\label{M21}
\sum_{i<j}x_i^2x_j
\end{equation}
that are not symmetric.
There is a ring $\QSy\supset\Sym$ of quasi-symmetric functions.
For any composition (ordered partition) $(a_1,\dots,a_k)$ of $n$,
the monomial symmetric function $M_{(a_1,\dots,a_k)}$ is the ``smallest''
quasi-symmetric function containing $x_1^{a_1}\cdots x_k^{a_k}$;
for example, the formal power series (\ref{M21}) is $M_{(2,1)}$.
Any quasi-symmetric function of degree $n$ is a linear combination 
of monomial quasi-symmetric functions of the same degree.
Any monomial symmetric function is a sum of monomial quasi-symmetric
functions, e.g.,
\[
m_{(1,1)}=M_{(1,1)},\quad m_{(2,1)}=M_{(2,1)}+M_{(1,2)} .
\]
\par
The ring $\Sym$ of symmetric functions is a polynomial ring in the
$e_k$, and also in the $h_k$, and also in the $p_k$.  In particular,
there are polynomials $P_n$ and $Q_n$ so that
\[
e_n=P_n(p_1,p_2,\dots,p_n)\quad\text{and}\quad h_n=Q_n(p_1,p_2,\dots,p_n) .
\]
In fact, these are exactly the polynomials that appear in the Introduction.
Explicit formulas are well-known.
\begin{prop}
For $n\ge 1$,
\begin{align}
\label{mform}
P_n(y_1,\dots,y_n)&=\sum_{m_1+2m_2+\cdots=n}\frac{(-1)^{m_2+m_4+\cdots}}{m_1!m_2!
\cdots} \left(\frac{y_1}{1}\right)^{m_1}\left(\frac{y_2}{2}\right)^{m_2}\cdots\\
Q_n(y_1,\dots,y_n)&=\sum_{m_1+2m_2+\cdots=n}\frac1{m_1!m_2!\cdots} 
\left(\frac{y_1}{1}\right)^{m_1}\left(\frac{y_2}{2}\right)^{m_2}\cdots .
\end{align}
\end{prop}
\begin{proof} If 
\begin{multline*}
E(t)=\sum_{n=0}^\infty e_nt^n=\prod_{i\ge 1}(1+tx_i),\quad
H(t)=\sum_{n=0}^\infty h_nt^n=\prod_{i\ge 1}\frac1{1-tx_i},\\
P(t)=\sum_{n=1}^\infty p_nt^{n-1}=\sum_{i\ge 1}\frac{x_i}{1-tx_i}
\end{multline*}
are the respective generating functions of the elementary, complete, and
power-sum symmetric functions, then evidently $E(t)=H(-t)^{-1}$ and
$H'(t)=P(t)H(t)$.  It follows that
\[
H(t)=\exp\left(\int_0^t P(s)ds\right)\quad\text{and}\quad
E(-t)=\exp\left(-\int_0^t P(s)ds\right) ,
\]
which can be expanded out to give the conclusion.
\end{proof}
\par\noindent
\begin{rem} The polynomials $P_k$ appear in \cite{S}, where they
are denoted $P_k/k!$.
The $Q_k$ appear in \cite{ChCh}, where they are denoted $\Omega_k$,
and in \cite{CoCa}, where they are denoted $P_k$.
\end{rem}
The polynomials $P_n$ and $Q_n$ can be written as determinants 
\cite[Ch. I \S2]{M}:
\[
n!P_n(y_1,\dots,y_n)=\left|\begin{matrix}y_1 & 1 & 0 & \dots & 0\\
y_2 & y_1 & 2 & \dots & 0\\
\vdots & \vdots & \vdots & \ddots & \vdots\\
y_{n-1} & y_{n-2} & y_{n-3} & \dots & n-1\\
y_n & y_{n-1} & y_{n-2} & \dots & y_1\end{matrix}\right| 
\]
and
\[
n!Q_n(y_1,\dots,y_n)=\left|\begin{matrix}y_1 & -1 & 0 & \dots & 0\\
y_2 & y_1 & -2 & \dots & 0\\
\vdots & \vdots & \vdots & \ddots & \vdots\\
y_{n-1} & y_{n-2} & y_{n-3} & \dots & -n+1\\
y_n & x_{n-1} & y_{n-2} & \dots & y_1\end{matrix}\right| .
\]
From these formulas can be deduced some properties of the polynomials 
$P_n$ and $Q_n$.
\begin{prop} 
Let $P_n,Q_n$ be as defined above.  Then
\begin{enumerate}
\item
All the coefficients of $n!P_n$ and $n!Q_n$ are integers.
\item
$n!P_n(a,a,\dots,a)=a(a-1)\cdots (a-n+1)$ and
$n!Q_n(a,a,\dots,a)=a(a+1)\cdots (a+n-1)$.
Hence the coefficients of $P_n$ sum to 0 for $n\ge 2$,
and the coefficients of $Q_n$ sum to 1.
\end{enumerate}
\end{prop}
\par
Let $a_1,\dots,a_n$ be a finite sequence of real constants.
There is a homomorphism $\QSy\to\R$ sending each $x_i$ with $i\le n$
to $a_i$, and each $x_i$ with $i>n$ to 0.
We denote the image of $u\in\QSy$ under this homomorphism by
$u(a_1,\dots,a_n)$.
\begin{lem}
\label{hn}
For positive integers $n$ and $k$,
\[
h_k(a_1,\dots,a_n,a_{n+1})=\sum_{j=0}^kh_{k-j}(a_1,\dots,a_n)a_n^j .
\]
\end{lem}
\begin{proof}
From the generating function $H(t)$ we have
\begin{multline*}
\sum_{j=0}^\infty t^jh_j(a_1,\dots,a_{n+1})=\prod_{i=1}^{n+1}\frac1{1-a_it}
=(1+a_{n+1}t+a_{n+1}^2t^2+\cdots)\prod_{i=1}^n\frac1{1-a_it}\\
=(1+a_{n+1}t+a_{n+1}^2t^2+\cdots)\sum_{j=0}^\infty t^jh_j(a_1,\dots,a_n),
\end{multline*}
and the conclusion follows by considering the coefficient of $t^k$.
\end{proof}
\par
As in \cite{H4}, let
\[
N_{n,m}=\sum_{\text{partitions $\la$ of $n$ with $m$ parts}}m_{\la}=
\sum_{\text{compositions $I$ of $n$ with $m$ parts}}M_I 
\]
for $n\ge m$.
The $N_{n,m}$ also appear in \cite[Ch. I, \S2, ex. 19]{M}, where
they are denoted $p_n^{(m)}$.
We shall need the following result in \S4.
\begin{lem}  
\label{ehN}
For $0\le j\le k$,
\[
e_jh_{k-j}=\sum_{p=j}^k \binom{p}{j}N_{k,p} .
\]
\end{lem}
\begin{proof}
Let
\[
\F(t,s)=1+\sum_{n\ge m\ge 1}N_{n,m}t^ns^m\in\Sym[[t,s]]
\]
be the generating function of the $N_{m,n}$.  Then
as in \cite[Lemma 1]{H4}, we have
\[
\F(t,s)=\prod_{i\ge 1}\left(1+\frac{stx_i}{1-tx_i}\right)
=\prod_{i\ge 1}\frac{1+(s-1)tx_i}{1-tx_i}
=H(t)E((s-1)t) .
\]
From this follows (cf. \cite[Lemma 2]{H4}, \cite[loc. cit.]{M})
\[
N_{k,j}=\sum_{p=0}^{k-j}\binom{j+p}{j}(-1)^je_{j+p}h_{k-j-p},
\]
and the system of equations for fixed $k$ can be backsolved to
give the conclusion.
\end{proof}
\section{Multiple zeta values}
Multiple zeta values $\zt(i_1,i_2,\dots,i_k)$ are defined by equation
(\ref{mzv}) above.
We refer to $i_1+\dots+i_k$ as the weight of this multiple zeta value,
and $k$ as its depth.  Multiple zeta values of arbitrary depth were
introduced by the author \cite{H1} and D. Zagier \cite{Z}, though the
depth 2 case was already studied by Euler \cite{E}.
\par
It is evident that multiple zeta values are related to 
quasi-symmetric functions.  In the notation introduced in the
last section,
\begin{equation}
\label{limit}
\zt(i_k,i_{k-1},\dots,i_1)=
\lim_{n\to\infty}M_{(i_1,\dots,i_k)}(1,\frac12,\dots,\frac1{n})
\end{equation}
for any composition $(i_1,\dots,i_k)$ with $i_k\ge 2$.
There is a subalgebra $\QSy^0$ of $\QSy$ generated by all the monomial 
quasi-symmetric functions $M_{(i_1,\dots,i_k)}$ with $i_k>1$, and we
can define a homomorphism $\zt:\QSy^0\to\R$ that sends $1\in\QSy^0$
to $1\in\R$ and $M_{(i_1,\dots,i_k)}$, $i_k>1$, to (\ref{limit})
(See \cite{H3} for a detailed discussion).
We write $\Sym^0$ for the subalgebra $\QSy^0\cap\Sym$ of $\Sym$; if we
think of $\Sym$ as the polynomial algebra on the $p_i$, then
$\Sym^0$ is the subalgebra generated by $p_2,p_3,\dots$.
\par
One of the first major results in the modern theory of multiple 
zeta values is the following ``sum theorem'', conjectured by C. Moen
(see \cite{H1}) and proved by A. Granville \cite{G}.
\begin{thm*}[Sum Theorem]
The sum of all multiple zeta values of weight $n$ and depth $d\le n-1$ is
$\zt(n)$.
\end{thm*}
Call a composition $(i_1,\dots,i_k)$ ``admissible'' if $i_1>1$.
The following result was proved in \cite{H1}.
\begin{thm*}[Derivation Theorem]
For any admissible compositon $(i_1,\dots,i_k)$,
\begin{multline*}
\sum_{j=1}^k\zt(i_1,\dots,i_{j-1},i_j+1,i_{j+1},\dots,i_k)=\\
\sum_{j=1}^k\sum_{p=1}^{i_j-1}\zt(i_1,\dots,i_{j-1},i_j-p+1,p,i_{j+1},\dots,i_k).
\end{multline*}
\end{thm*}
Another important result on multiple zeta values is the duality
theorem.  It was actually conjectured in \cite{H1}, but the proof
comes easily from a description of multiple zeta values as iterated
integrals (see \cite{Z} or \cite{H2}).  To describe it requires some 
definitions.
Let $\Si$ be the function that takes a composition to its sequence
of partial sums:
\[
\Si(i_1,\dots,i_k)=(i_1,i_1+i_2,\dots,i_1+\dots+i_k).
\]
On the set $I_n$ of increasing integer sequences chosen from the set 
$\{1,\dots,n\}$, there are functions $R_n:I_n\to I_n$ and $C_n:I_n\to I_n$ 
given by
\begin{align*}
R_n(s_1,\dots,s_k)&=(n+1-s_k,n+1-s_{k-1},\dots,n+1-s_1)\\
C_n(s_1,\dots,s_k)&=\text{complement of $\{s_1,\dots,s_k\}$ in $\{1,\dots,n\}$}
\end{align*}
For a composition $(i_1,\dots,i_k)$ of weight $n$, let
\[
\tau(i_1,\dots,i_k)=\Si^{-1}R_nC_n\Si(i_1,\dots,i_k).
\]
Then $\tau(I)$ is admissible if $I$ is, and we have the following result.
\begin{thm*}[Duality Theorem] For any admissible composition $I$, 
$\zt(\tau(I))=\zt(I)$.
\end{thm*}
As shown in \cite{H1}, if $I$ and $J$ are admissible compositions then
their juxtaposition $IJ$ has the property that $\tau(IJ)=\tau(J)\tau(I)$.
Since the composition $I=(2)$ is self-dual, $\tau(I)$ ends in 1 if and only
if $I$ doesn't begin with 2.
\par
We note that multiple zeta values of depth greater than 1 cannot
in general be written as rational polynomials in the depth 1
(ordinary) zeta values:  for example, there is no such expression known 
for $\zt(2,6)$.  It is true, however, that all multiple zeta values
of weight 7 or less can be so expressed.  Euler \cite{E} gave the formula
\begin{equation}
\label{euln1}
\zt(n,1)=\frac{n}2\zt(n+1)+\frac12\sum_{i=1}^{n-2}\zt(n-i)\zt(i+1)
\end{equation}
valid for $n\ge 2$, and if $a+b$ is odd the double zeta value
$\zt(a,b)$ can be written as a rational polynomial in the $\zt(i)$.
Multiple zeta values of ``height one'', i.e., those of the form
$\zt(n,1,\dots,1)$, are also rational polynomials in the $\zt(i)$,
as can be seen from the generating function \cite{H2}
\[
\sum_{m,n\ge 1}s^mt^n\zt(m+1,\underbrace{1,\dots,1}_{n-1})=
1-\exp\left(\sum_{j\ge 2}\frac{\zt(j)}{j}(s^j+t^j-(s+t)^j)\right) .
\]
We also note that all known identities of multiple zeta values 
preserve weight.
\section{$H$-functions and summation formulas}
Now we show how to obtain families of summation formulas like those 
given in the Introduction.  For $u\in\QSy$ and nonnegative integers
$s_1,\dots,s_k$ with $s_1+\dots+s_k\ge 2$, define the $H$-function
$\eta_{s_1,\dots,s_k}:\QSy\to\R$ by
\begin{equation}
\label{eta}
\eta_{s_1,\dots,s_k}(u)=
\sum_{n=1}^\infty\frac{u(1,\frac12,\dots,\frac1n)}{n^{s_1}(n+1)^{s_2}\cdots
(n+k-1)^{s_k}}.
\end{equation}
Then we have the following result.
\begin{thm}
\label{etacon}
$\eta_{s_1,\dots,s_k}(u)$ converges for any $u\in\QSy$. 
\end{thm}
\begin{proof}
It suffices to show that $\eta_{s_1,\dots,s_k}(M_I)$ converges for 
any composition $I$.  Writing $I=(i_1,\dots,i_j)$, we have
\begin{align*}
\eta_{s_1,\dots,s_k}(M_I)&=
\sum_{n=1}^\infty\frac1{n^{s_1}(n+1)^{s_2}\cdots (n+k-1)^{s_k}}
\sum_{1\le n_1<\dots<n_j\le n}\frac1{n_1^{i_1}\cdots n_j^{i_j}}\\
&=\sum_{1\le n_1<\cdots <n_j}\frac1{n_1^{i_1}\cdots n_j^{i_j}}
\sum_{m=n_j}^\infty\frac1{m^{s_1}(m+1)^{s_2}\cdots (m+k-1)^{s_k}} .
\end{align*}
Now the terms in the latter sum are evidently bounded above by
those for $\eta_2(M_I)$.  But 
\begin{multline}
\label{eta2}
\eta_2(M_I)=\sum_{1\le n_1<\cdots <n_j}\frac1{n_1^{i_1}\cdots n_j^{i_j}}
\sum_{m=n_j}^\infty\frac1{m^2}=\\
\zt(i_j+2,i_{j-1},\dots,i_1)+\zt(2,i_j,\dots,i_1),
\end{multline}
which converges.
(In the case $I=\emptyset$, equation (\ref{eta2}) should be
interpreted as $\eta_2(1)=\zt(2)$.  In general $\eta_{s_1,\dots,s_k}(1)$
is the ``$H$-series'' $H(s_1,\dots,s_k)$ discussed in \cite{HM}.)
\end{proof}
In this section we shall be concerned with the examples 
$\eta_2$, $\eta_{1,1}$, and $\eta_{0,1,1}$.  
We already have equation (\ref{eta2}) for $\eta_2$.
For the other two functions we have the following result.
\begin{thm}
\label{nome}
For any composition $I=(i_1,\dots,i_j)$,
\begin{align}
\label{eta11}
\eta_{1,1}(M_I)&=\zt(i_j+1,i_{j-1},\dots,i_1),\\
\label{eta011}
\eta_{0,1,1}(M_I)&=\begin{cases} 1,&\text{if $I=(1)$,}\\
\eta_{0,1,1}(M_{(i_1,\dots,i_{j-1})}),&\text{if $i_j=1$ and $j\ge 2$,}\\
\zt(i_j,\dots,i_1)-\eta_{0,1,1}(M_{(i_1,\dots,i_{j-1},i_j-1)}),&\text{otherwise.}
\end{cases}
\end{align}
\end{thm}
\begin{proof}
Since
\[
\sum_{m=n}^\infty\frac1{m(m+1)}=\frac1{n}
\]
it follows that
\begin{multline*}
\eta_{1,1}(M_I)=\sum_{1\le n_1<\cdots <n_j}\frac1{n_1^{i_1}\cdots n_j^{i_j}}
\sum_{m=n_j}^\infty\frac1{m(m+1)}\\
=\sum_{1\le n_1<\cdots <n_j}\frac1{n_1^{i_1}\cdots n_j^{i_j+1}}
=\zt(i_j+1,i_{j-1},\dots,i_1) .
\end{multline*}
We have also
\begin{multline*}
\eta_{0,1,1}(M_I)=\sum_{1\le n_1<\dots<n_j}\frac1{n_1^{i_1}\cdots n_j^{i_j}}
\sum_{m=n_j+1}^\infty\frac1{m(m+1)}\\
=\sum_{1\le n_1<\dots<n_j}\frac1{n_1^{i_1}\cdots n_j^{i_j}(n_j+1)} .
\end{multline*}
If $I=(1)$, this is 
\[
\eta_{0,1,1}(p_1)=\sum_{n=1}^\infty \frac1{n(n+1)}=1.
\]
Now suppose $I\ne (1)$.  If $i_j=1$, we  have
\begin{multline*}
\eta_{0,1,1}(M_I)=\sum_{1\le n_1<\dots<n_j}
\frac1{n_1^{i_j}\cdots n_{j-1}^{i_{j-1}}n_j(n_j+1)}=\\
\sum_{1\le n_1<\dots<n_{j-1}}\frac1{n_1^{i_1}\cdots n_{j-1}^{i_{j-1}}}\sum_{m=n_{j-1}+1}^\infty
\frac1{m(m+1)}\\
=\sum_{1\le n_1<\dots<n_{j-1}}\frac1{n_1^{i_1}\cdots n_{j-1}^{i_{j-1}}(n_{j-1}+1)}
=\eta_{0,1,1}(M_{(i_1,\dots,i_{j-1})}) .
\end{multline*}
On the other hand, if $i_j>1$ we have
\begin{multline*}
\eta_{0,1,1}(M_I)=
\sum_{1\le n_1<n_2<\dots<n_j}\frac1{n_1^{i_1}n_2^{i_2}\cdots n_{j-1}^{i_{j-1}}n_j^{i_j-1}}
\left(\frac1{n_j}-\frac1{n_j+1}\right)=\\
\zt(i_j,i_{j-1},\dots,i_1)-\eta_{0,1,1}(M_{(i_1,\dots,i_{j-1},i_j-1)}) ,
\end{multline*}
and the conclusion follows.
\end{proof}
Now we consider the images under $\eta_2$, $\eta_{1,1}$ and $\eta_{0,1,1}$
of the symmetric functions $N_{n,k}$ introduced in \S2 above.
For integers $n\ge 2$ and $1\le k< n$, let $S_{n,k}$ be the sum of
all multiple zeta values of weight $n$ and depth $k$:  by the
sum theorem for multiple zeta values, $S_{n,k}=\zt(n)$.
We can also write $S_{n,k}=S_{n,k}^{[2]}+S_{n,k}^T$, where
\[
S_{n,k}^{[2]}=\sum_{2+i_2+\dots+i_k=n}\zt(2,i_2,\dots,i_k) .
\]
Note that $S_{n,n-1}^{[2]}=S_{n,n-1}=\zt(n)$, and that $S_{n,1}^T=S_{n,1}=\zt(n)$
for $n>2$.
The following fact is an immediate consequence of equation (\ref{eta2}).
\begin{prop}
$\eta_2(N_{n,k})=S_{n+2,k}^T+S_{n+2,k+1}^{[2]}$.
\end{prop}
Similarly, Theorem \ref{nome} gives the following.
\begin{prop}
\label{eta11n}
$\eta_{1,1}(N_{n,k})=S_{n+1,k}=\zt(n+1)$.
\end{prop}
\begin{prop}  $\eta_{0,1,1}(N_{n,n})=1$, and if $n\ge 2$,
\[
\eta_{0,1,1}(N_{n,k})=\begin{cases} \zt(n)-\eta_{0,1,1}(N_{n-1,1}),& k=1,\\
\zt(n)+\eta_{0,1,1}(N_{n-1,k-1}-N_{n-1,k}),& 1<k<n.
\end{cases}
\]
\end{prop}
\begin{proof} Note that $N_{n,n}=e_n$, and by applying
equation (\ref{eta011}) repeatedly, we have
\begin{equation}
\label{ome}
\eta_{0,1,1}(e_n)=\eta_{0,1,1}(e_{n-1})=\dots=\eta_{0,1,1}(e_1)=1.
\end{equation}
The second statement is immediate from (\ref{eta011}).
\end{proof}
Note that equation (\ref{ch1}) follows from the case $k=n$ of
the preceding result.  In the case $k=1$, the result is
\begin{multline}
\label{omp}
\eta_{0,1,1}(p_n)=\zt(n)-\eta_{0,1,1}(N_{n-1,1})=\dots=\\
\zt(n)-\zt(n-1)+\dots+(-1)^n\zt(2)+(-1)^{n+1}.
\end{multline}
Equation (\ref{omp}) implies the following result,
which appears in the remark following \cite[Theorem 2.1]{So}.
\begin{cor}
For $k\ge 2$,
\[
\sum_{n=1}^\infty\frac{H_n^{(k)}}{(n+1)(n+2)}=
\sum_{i=0}^{k-2}(-1)^i\zt(k-i)+(-1)^{k-1} .
\]
\end{cor}
\par
Now we in a position to prove Theorems \ref{ch2} through \ref{eta2pq} of the
Introduction by finding the images under the 
three $H$-functions $\eta_2$, $\eta_{1,1}$, and $\eta_{0,1,1}$
of the symmetric function $e_jh_{n-j}$.
First we consider $\eta_2(e_jh_{n-j})$.
\begin{thm}  Let $0\le j\le n$.  Then
\label{eta2eh}
\[
\eta_2(e_jh_{n-j})=\begin{cases} (n+1)\zt(n+2), & j=0,\\
\sum_{p=j}^n\binom{p-1}{j-1}S_{n+2,p}^T+\binom{n+1}{j+1}\zt(n+2),& j\ge 1.
\end{cases}
\]
\end{thm}
\begin{proof} If $j=0$ we have
\begin{multline*}
\eta_2(h_n)=\eta_2\left(\sum_{k=1}^n N_{n,k}\right)=
\sum_{k=1}^n(S_{n+2,k}^T+S_{n+2,k+1}^{[2]})\\
=S_{n+2,1}+\sum_{k=2}^n(S_{n+2,k}^{[2]}+S_{n+2,k}^T)+S_{n+2,n+1}=
(n+1)\zt(n+2).
\end{multline*}
Now suppose $j\ge 1$.  Then using Lemma \ref{ehN},
\begin{align*}
\eta_2(e_jh_{n-j})&=\sum_{k=j}^n\binom{k}{j}\eta_2(N_{n,k})\\
&=\sum_{k=j}^n\binom{k}{j}(S_{n+2,k}^T+S_{n+2,k+1}^{[2]})\\
&=S_{n+2,j}^T+\sum_{k=j+1}^n\left[\binom{k-1}{j}S_{n+2,k}^{[2]}+
\binom{k}{j}S_{n+2,k}^T\right]+\binom{n}{j}S_{n+2,n+1}\\
&=S_{n+2,j}^T+\sum_{k=j+1}^n\left[\binom{k-1}{j}\zt(n+2)+\binom{k-1}{j-1}S_{n+2,k}^T
\right]+\binom{n}{j}\zt(n+2)\\
&=\sum_{k=j}^n\binom{k-1}{j-1}S_{n+2,k}^T+\binom{n+1}{j+1}\zt(n+2) .
\end{align*}
\end{proof}
\begin{proof3}
Equation (\ref{qn2}) follows from the case $j=0$ of the preceding result.
By duality of multiple zeta values, $S_{n,k}^T=S_{n,n-k}^R$, where
$S_{n,k}^R$ is the sum of all weight-$n$, depth $k$ multiple zeta values
whose exponent string ends in 1.  From the preceding result we have
\[
\eta_2(e_n)=S_{n+2,n}^T+\zt(n+2)=S_{n+2,2}^R+\zt(n+2),
\]
from which follows
\[
\eta_2(e_n)=\zt(n+1,1)+\zt(n+2)
=\frac{n+3}{2}\zt(n+2)+\frac12\sum_{j=2}^n\zt(j)\zt(n+2-j) .
\]
using Euler's formula (\ref{euln1}), and thus equation (\ref{pn2}).
\qed
\end{proof3}
We also have the following result.
\begin{cor} For $n\ge 2$,
\[
\eta_2(e_{n-1}h_1)=\zt(n,2)+n\zt(n+1,1)+(n+1)\zt(n+2) .
\]
\end{cor}
\begin{proof}
We have
\begin{align*} 
\eta_2(e_{n-1}h_1)&=S_{n+2,n-1}^T+(n-1)S_{n+2,n}^T +(n+1)\zt(n+2)\\
&=S_{n+2,3}^R+(n-1)S_{n+2,2}^R+(n+1)\zt(n+2)\\
&=\sum_{j=2}^n\zt(j,n+1-j,1)+(n-1)\zt(n+1,1)+(n+1)\zt(n+2)\\
&=\zt(n,2)+\zt(n+1,1)+(n-1)\zt(n+1,1)+(n+1)\zt(n+2),
\end{align*}
where we have used the derivation theorem for multiple zeta values in 
the last step.
\end{proof}
In general the formula for $\eta_2(e_jh_{n-j})$ given by Theorem \ref{eta2eh} 
cannot be reduced to ordinary zeta values if $1\le j<n$, unless
$n\le 5$.  For example,
\begin{align*}
\eta_2(e_2h_2)&=10\zt(6)+S_{6,2}^T+2S_{6,3}^T+3S_{6,4}^T\\
&=11\zt(6)+3S_{6,2}^R-S_{6,2}^{[2]}+2S_{6,3}^R\\
&=11\zt(6)+3\zt(5,1)-\zt(2,4)+2(\zt(4,1,1)+\zt(3,2,1)+\zt(2,3,1))\\
&=11\zt(6)+3\zt(5,1)-\zt(2,4)+2(\zt(5,1)+\zt(4,2))\\
&=11\zt(6)+5\zt(5,1)+2\zt(4,2)-\zt(2,4)\\
&=10\zt(6)+\frac12\zt(3)^2 ,
\end{align*}
but
\[
\zt(e_5h_1)=\zt(6,2)+6\zt(7,1)+7\zt(8)=\zt(6,2)+6\zt(3)\zt(5)+\frac{83}2\zt(8)
\]
has no known expression as a rational polynomial in the $\zt(i)$ since
$\zt(6,2)$ doesn't.
\par
\begin{proof2}
By Lemma \ref{ehN} and Proposition \ref{eta11n}
\[
\eta_{1,1}(e_jh_{n-j})=\sum_{k=j}^n\binom{k}{j}\eta_{1,1}(N_{n,k})=
\sum_{k=j}^n\binom{k}{j}\zt(n+1),
\]
and the conclusion follows.
\end{proof2}
\begin{proof1}
It is enough to show that
\[
\eta_{0,1,1}(e_jh_{n-j})=
\begin{cases}
\sum_{k=0}^j\binom{n-k}{j+1-k}\zt(n-k)-\zt(n-j),& j\le n-2,\\
\sum_{k=0}^{n-2}\zt(n-k)+1,&j=n-1,\\
1,&j=n,\end{cases}
\]
for $n\ge 2$.
Recall from equation (\ref{ome}) that $\eta_{0,1,1}(e_n)=1$, so the
result is true for $j=n$.
If $j=0$, we have
\begin{multline*}
\eta_{0,1,1}(h_n)=\sum_{k=1}^n \eta_{0,1,1}(N_{n,k})
=S_{n,1}-\eta_{0,1,1}(N_{n-1,1})+\eta_{0,1,1}(N_{n-1,1})+S_{n,2}\\
-\eta_{0,1,1}(N_{n-1,1})
+\dots+\eta_{0,1,1}(N_{n-1,n-2})+S_{n,n-1}-\eta_{0,1,1}(N_{n-1,n-1})+1\\
=S_{n,1}+S_{n,2}+\dots+S_{n,n-1}=(n-1)\zt(n)
\end{multline*}
and again the result holds.  Now let $0<j<n$.  Then
\begin{align*}
\eta_{0,1,1}(e_jh_{n-j})&=\sum_{k=j}^n\binom{k}{j}\eta_{0,1,1}(N_{n,k})\\
&=\sum_{k=j}^{n-1}\binom{k}{j}[\zt(n)+\eta_{0,1,1}(N_{n-1,k-1}-N_{n-1,k})]+\binom{k}{j}\\
&=\sum_{k=j}^n\binom{k}{j}\zt(n)+\eta_{0,1,1}(N_{n-1,j-1})+\\
\sum_{k=j}^{n-2}&\left[\binom{k+1}{j}-\binom{k}{j}\right]\eta_{0,1,1}(N_{n-1,k})-
\binom{n-1}{j}+\binom{n}{j}\\
&=\binom{n}{j+1}\zt(n)+\sum_{k=j-1}^{n-1}\binom{k}{j-1}\eta_{0,1,1}(N_{n-1,k})\\
&=\binom{n}{j+1}\zt(n)+\eta_{0,1,1}(e_{j-1}h_{n-1}),
\end{align*}
and the result follows by induction on $j$.
\qed
\end{proof1}
\section{Further summation formulas}
We return to the general $H$-functions $\eta_{s_1,\dots,s_k}$ defined by equation
(\ref{eta}).  If $s_k>0$, we call $k$ the length of $\eta_{s_1,\dots,s_k}$.
We have the following result (cf. \cite[Lemma 1]{HM}).
\begin{prop}
\label{diff}
Let $s_1,\dots,s_k$ be a nonnegative integer sequence with $s_i,s_j\ge 1$
for $1\le i<j\le k$.  
If $s_1+\dots+s_k\ge 3$, then
\[
\eta_{s_1,\dots,s_i,\dots,s_j,\dots,s_k}=
\frac1{j-i}(\eta_{s_1,\dots,s_i,\dots,s_j-1,\dots,s_k}-\eta_{s_1,\dots,s_i-1,\dots,s_j,\dots,s_k}).
\]
\end{prop}
\begin{proof}
This follows immediately from the definition and 
\[
\frac1{(n+i-1)(n+j-1)}=\frac1{j-i}\left[\frac1{n+i-1}-\frac1{n+j-1}\right].
\]
\end{proof}
\subsection{Length 2}
The following result is immediate from Proposition \ref{diff}.
\begin{prop}
Any $H$-function of length 2 can be written as a rational linear 
combination of the functions $\eta_p$, $\eta_{0,p}$, $p\ge 2$, and $\eta_{1,1}$.
\end{prop}
For example, since $\eta_{2,1}=\eta_2-\eta_{1,1}$ we have
\begin{multline*}
\sum_{n=1}^\infty \frac{H_n^2}{n^2(n+1)}=\eta_2(p_1^2)-\eta_{1,1}(p_1^2)
=\eta_2(p_2)+2\eta_2(e_2)-\eta_{1,1}(p_2)-2\eta_{1,1}(e_2)\\
=\zt(2,2)+\zt(4)+2\zt(4)+2\zt(3,1)-\zt(3)-2\zt(3)=\frac{17}4\zt(4)-3\zt(3).
\end{multline*}
\par
In the preceding section we gave formulas for the values of 
$\eta_2$ and $\eta_{1,1}$ on monomial quasi-symmetric functions $M_I$.
For $\eta_p$ and $\eta_{0,p}$ we have the following result.
\begin{prop}
\label{pwer}
Let $I=(i_1,\dots,i_j)$ be a composition.  If $p\ge 2$, then
\begin{align*}
\eta_p(M_I)&=\zt(p,i_j,\dots,i_1)+\zt(p+i_j,i_{j-1},\dots,i_1)\\
\eta_{0,p}(M_I)&=\zt(p,i_j,\dots,i_1) .
\end{align*}
\end{prop}
\begin{proof}
This is immediate from the equation
\begin{multline*}
\eta_{s_1,\dots,s_k}(M_{(i_1,\dots,i_j}))=\\
\sum_{1\le n_1<n_2<\dots<n_j}\frac1{n_1^{i_1}\cdots n_j^{i_j}}
\sum_{m=n_j}^\infty\frac1{m^{s_1}(m+1)^{s_2}\cdots (m+k-1)^{s_k}}
\end{multline*}
appearing in the proof of Theorem \ref{etacon}.
\end{proof}
This result has the following corollary, special cases of which have
appeared in the literature.
Special cases of the first equation appear many places, the case
$k=2$ of the second appears as \cite[eqn. (2.5c)]{Zh}, and the
third equation can be deduced from \cite[Corollary 1]{CoCa}.
\begin{cor}
\label{eta3peh}
For $k\ge 1$,
\begin{align*}
\sum_{n=1}^\infty \frac{H_n^{(k)}}{n^3}&=\zt(k+3)+\zt(3,k)\\
\sum_{n=1}^\infty \frac{P_k(H_n,H_n^{(2)},\dots,H_n^{(k)})}{n^3}&=
\zt(k+2,1)+\zt(k+1,1,1) \\
\sum_{n=1}^\infty \frac{Q_k(H_n,H_n^{(2)},\dots,H_n^{(k)})}{n^3}&=
\zt(k+3)+\sum_{j=2}^{k+1}S_{k+3,j}^T .
\end{align*}
\end{cor}
\begin{proof} 
In each case, apply the first part of Proposition \ref{pwer} with
$p=3$.
\end{proof}
Another corollary with many special cases in the literature is the
following.
\begin{cor}
\label{eta02pe}
For $k\ge 1$, 
\[
\sum_{n=1}^\infty \frac{H_n^{(k)}}{(n+1)^2}=\zt(2,k)\quad\text{and}\quad
\sum_{n=1}^\infty \frac{P_k(H_n,H_n^{(2)},\dots,H_n^{(k)})}{(n+1)^2}=\zt(k+2) .
\]
\end{cor}
\begin{proof}
Similar to that for the preceding corollary, using the second part
of Proposition \ref{pwer}.
\end{proof}
\par\noindent
\begin{rem} In the case $k=1$, both equations give
\[
\sum_{n=1}^\infty\frac{H_n}{(n+1)^2}=\zt(3),
\]
which appears as \cite[eqn. (1.1a)]{Ch}.
In the case $k=2$, the two equations give
\[
\sum_{n=1}^\infty\frac{H_n^{(2)}}{(n+1)^2}=\zt(2,2)=\frac34\zt(4)
\quad\text{and}\quad
\sum_{n=1}^\infty\frac{H_n^2-H_n^{(2)}}{(n+1)^2}=2\zt(4) ,
\]
from which follows
\[
\sum_{n=1}^\infty\frac{H_n^2}{(n+1)^2}=\frac{11}4\zt(4).
\]
Cf. \cite[eqn. (2)]{BB}, \cite[eqns. (1.2b),(1.4a),(1.5b)]{Ch}, 
and \cite[eqns. (24a),(24b)]{ChCh}.  
In the case $k=3$, the second equation gives
\[
\sum_{n=1}^\infty\frac{H_n^3-3H_nH_n^{(2)}+2H_n^{(3)}}{(n+1)^2}=6\zt(5),
\]
which is \cite[eqn. (2.3a)]{Zh}.
\end{rem}
There is not in general a nice formula for $\eta_{0,2}(h_k)$ or
for $\eta_3(h_k)$, but we do have the following result.
\begin{prop} $\eta_{0,2}(h_{k+1})+\eta_3(h_k)=(k+2)\zt(k+3)$.
\end{prop}
\begin{proof}
From the third equation of Corollary \ref{eta3peh} we have
\[
\eta_3(h_k)=\zt(k+3)+S_{k+3,2}^T+\dots+S_{k+3,k+1}^T ,
\]
and from Proposition \ref{pwer}
\[
\eta_{0,2}(h_{k+1})=\eta_{0,2}(N_{k+1,1}+\dots+N_{k+1,k+1})=
S_{k+3,2}^{[2]}+S_{k+3,3}^{[2]}+\dots+S_{k+3,k+2}^{[2]} .
\]
Since $S_{k+3,k+2}^{[2]}=\zt(k+3)$, these two equations can be added 
to obtain the conclusion.
\end{proof}
The preceding result can be written
\begin{multline*}
\sum_{n=1}^\infty\frac{Q_{k+1}(H_n,H_n^{(2)},\dots,H_n^{(k+1)})}{(n+1)^2}+
\sum_{n=1}^\infty\frac{Q_k(H_n,H_n^{(2)},\dots,H_n^{(k)})}{n^3}\\
=(k+2)\zt(k+3) .
\end{multline*}
\par
The second equation of Corollary \ref{eta02pe} can be generalized
as follows.
\begin{thm}
\label{eta02eh}
For nonnegative integers $l,k$,
\begin{multline*}
\sum_{n=0}^\infty \frac{Q_l(H_n,H_n^{(2)},\dots,H_n^{(l)})P_k(H_n,H_n^{(2)},\dots,
H_n^{(k)})}{(n+1)^2}\\
=\binom{l+k+1}{k+1}\zt(l+k+2)-\sum_{p=k}^{l+k-1}\binom{p}{k}S_{l+k+2,l+k+1-p}^R .
\end{multline*}
\end{thm}
\begin{proof} 
From Proposition \ref{pwer} we have $\eta_{0,2}(N_{n,k})=S_{n+2,k+1}^{[2]}$.
Hence, using Lemma \ref{ehN} and the sum theorem for multiple zeta
values,
\begin{multline*}
\eta_{0,2}(e_kh_l)=\sum_{p=k}^{l+k}\binom{p}{k}\eta_{0,2}(N_{l+k,p})
=\sum_{p=k}^{l+k}\binom{p}{k}S_{l+k+2,p+1}^{[2]}\\
=\sum_{p=k}^{l+k-1}\binom{p}{k}S_{l+k+2,p+1}^{[2]}+\binom{l+k}{k}\zt(l+k+2)\\
=\sum_{p=k}^{l+k-1}\binom{p}{k}(\zt(l+k+2)-S_{l+k+2,p+1}^T)
+\binom{l+k}{k}\zt(l+k+2)\\
=\left(\binom{l+k}{k+1}+\binom{l+k}{k}\right)\zt(l+k+2)
-\sum_{p=k}^{l+k-1}\binom{p}{k}S_{l+k+2,l+k+1-p}^R,
\end{multline*}
and the result follows.
\end{proof}
From this we can deduce the following result.  The special case
$k=2$ of the first equation appears as \cite[eqn. (2.3b)]{Zh}.
\begin{cor} 
\label{eta02sc}
For $k\ge 1$,
\begin{align*}
\sum_{n=0}^\infty \frac{H_nP_k(H_n,H_n^{(2)},\dots,H_n^{(k)})}{(n+1)^2}
&=(k+2)\zt(k+3)-\zt(k+2,1) \\
\frac12\sum_{n=1}^\infty \frac{(H_n^2+H_n^{(2)})P_k(H_n,H_n^{(2)},\dots,H_n^{(k)})}
{(n+1)^2}&=\binom{k+3}{2}\zt(k+4)\\
&-(k+2)\zt(k+3,1)-\zt(k+2,2) .
\end{align*}
\end{cor}
\begin{rem}
From the case $k=2$ of the second equation follows
\[
\sum_{n=1}^\infty\frac{H_n^4}{(n+1)^2}=\frac{100}3\zt(6)+4\zt(3)^2+\zt(2,4)
+2\zt(2,2,2)=\frac{859}{24}\zt(6)+3\zt(3)^2,
\]
which was stated as a conjecture in \cite{Cof}.
\end{rem}
\begin{proof4}
From Lemma \ref{hn} it follows that
\[
Q_{l}(H_{n+1},H_{n+1}^{(2)},\dots,H_{n+1}^{(l)})=\sum_{j=0}^l
\frac{Q_j(H_n,H_n^{(2)},\dots,H_n^{(j)})}{(n+1)^{l-j}}
\]
and so
\begin{multline*}
\sum_{n=0}^\infty \frac{Q_l(H_{n+1},H_{n+1}^{(2)},\dots,H_{n+1}^{(l)})
P_k(H_n,H_n^{(2)},\dots,H_n^{(k)})}{(n+1)^2}=\\
\sum_{j=0}^l\sum_{n=1}^\infty\frac{Q_j(H_n,H_n^{(2)},\dots,H_n^{(j)})
P_k(H_n,H_n^{(2)},\dots,H_n^{(k)})}{(n+1)^{2+l-j}}.
\end{multline*}
Comparing this with Theorem \ref{eta02eh}, we see that to prove
the result it suffices to show
\[
\sum_{j=0}^{l-1}\sum_{n=0}^\infty\frac{Q_j(H_n,H_n^{(2)},\dots,H_n^{(j)})
P_k(H_n,H_n^{(2)},\dots,H_n^{(k)})}{(n+1)^{2+l-j}}=
\sum_{p=k}^{l+k-1}\binom{p}{k}S_{l+k+2,p+1}^T ,
\]
or
\begin{equation}
\label{need}
\sum_{j=3}^{l+2}\eta_{0,j}(e_kh_{2+l-j})=
\sum_{p=k}^{l+k-1}\binom{p}{k}S_{l+k+2,p+1}^T .
\end{equation}
Using Lemma \ref{ehN} and Proposition \ref{pwer}, the left-hand side of 
(\ref{need}) is
\[
\sum_{j=3}^{l+2}\sum_{p=k}^{k+l+2-j}\binom{p}{k}\eta_{0,j}(N_{k+l+2-j,p})=
\sum_{j=3}^{l+2}\sum_{p=k}^{k+l+2-j}\binom{p}{k}S_{k+l+2,p+1}^{[j]},
\]
where $S_{n,k}^{[j]}$ is the sum of all weight $n$, depth $k$ multiple
zeta values whose exponent string starts with $j$.  This can be
rearranged as
\[
\sum_{p=k}^{k+l-1}\binom{p}{k}\sum_{j=3}^{k+l+2-p}S_{k+l+2,p+1}^{[j]}=
\sum_{p=k}^{k+l-1}\binom{p}{k}S_{k+l+2,p+1}^T ,
\]
and equation (\ref{need}) follows.
\qed
\end{proof4}
\subsection{Length 3}
For length 3 $H$-functions, Proposition \ref{diff} gives the following
result.
\begin{prop}
Any $H$-function of length 3 can be written as a rational linear 
combination of the 
functions $\eta_p$, $\eta_{0,p}$, $\eta_{0,0,p}$, $p\ge 2$, $\eta_{1,1}$, 
and $\eta_{0,1,1}$.
\end{prop}
For example, we can prove Theorem \ref{eta111} by writing $\eta_{1,1,1}$
as a linear combination of previously studied functions.
\begin{proof5}
From Proposition \ref{diff}, $\eta_{1,1,1}=\frac12(\eta_{1,1}-\eta_{0,1,1})$.
Hence we can use Theorems \ref{ch2} and \ref{qpnn1} to get
\begin{multline*}
\eta_{1,1,1}(e_jh_{n-j})\\
=\begin{cases}
\frac12\left[\binom{n+1}{j+1}\zt(n+1)-\sum_{k=0}^j\binom{n-k}{j+1-k}
\zt(n-k)-\zt(n-j)\right],& j\le n-2,\\
\frac12\left[(n+1)\zt(n+1)-\sum_{k=0}^{n-2}\zt(n-k)-1\right],& j=n-1,\\
\frac12(\zt(n+1)-1),& j=n,
\end{cases}
\end{multline*}
from which the conclusion follows.
\qed
\end{proof5}
We have developed formulas for $\eta_{1,1}$, $\eta_{0,1,1}$, $\eta_p$,
and $\eta_{0,p}$ of a monomial quasi-symmetric function $M_I$.
For $\eta_{0,0,p}(M_I)$ we have the following result.
\begin{prop}
If $p\ge 2$, then
\[
(\eta_{0,p}-\eta_{0,0,p})(M_I)=
\sum_{1\le n_1<n_2<\dots<n_j}\frac1{n_1^{i_1}\cdots n_j^{i_j}(n_j+1)^p} 
\]
for any composition $I=(i_1,\dots,i_j)$.
\end{prop}
In particular, we have the following.
\begin{prop} Let $T=\eta_{0,2}-\eta_{0,0,2}$.  Then $T(p_1)=2-\zt(2)$, and 
for a composition $I=(i_1,\dots,i_j)\ne (1)$, 
\begin{multline*}
T(M_I)=\\
\begin{cases}
\zt(i_j,\dots,i_1)-\eta_{0,1,1}(M_{(i_1,\dots,i_{j-1},i_j-1)})-
T(M_{(i_1,\dots,i_{j-1},i_j-1)}),& i_j>1,\\
\eta_{0,1,1}(M_{(i_1,\dots,i_{j-1})})-\zt(2,i_{j-1},\dots,i_1)+
T(M_{(i_1,\dots,i_{j-1})}),& i_j=1.\end{cases}
\end{multline*}
\end{prop}
\begin{proof}
The statement about $(\eta_{0,2}-\eta_{0,0,2})(p_1)$ follows
immediately from the preceding result.  Also, if $I\ne(1)$ we have
\begin{multline*}
(\eta_{0,2}-\eta_{0,0,2})(M_I)=\sum_{1\le n_1<\dots<n_j}\frac1{n_1^{i_1}\cdots
n_j^{i_j}(n_j+1)^2}\\
=\sum_{1\le n_1<\dots<n_j}\frac1{n_1^{i_1}\cdots n_j^{i_j}(n_j+1)}
-\sum_{1\le n_1<\dots<n_j}\frac1{n_1^{i_1}\cdots n_j^{i_j-1}(n_j+1)^2} .
\end{multline*}
If $n_j>1$, this is
\begin{multline*}
\eta_{0,1,1}(M_{(i_1,\dots,i_j)})
-(\eta_{0,2}-\eta_{0,0,2})(M_{(i_1,\dots,i_{j-1},i_j-1)})=\\
\zt(i_j,\dots,i_{j-1})-\eta_{0,1,1}(M_{(i_1,\dots,i_{j-1},i_j-1)})
-(\eta_{0,2}-\eta_{0,0,2})(M_{(i_1,\dots,i_{j-1},i_j-1)})
\end{multline*}
using Theorem \ref{nome}.  If $n_j=1$, $(\eta_{0,2}-\eta_{0,0,2})(M_I)$ is
\begin{multline*}
\eta_{0,1,1}(M_{(i_1,\dots,i_{j-1},1)})-\sum_{1\le n_1<\dots<n_j}
\frac1{n_1^{i_1}\cdots n_{j-1}^{i_{j-1}}(n_j+1)^2}=\\
\eta_{0,1,1}(M_{(i_1,\dots,i_{j-1})})-\zt(2,i_{j-1},\dots,i_1)+\sum_{1\le n_1<\dots<n_{j-1}}
\frac1{n_1^{i_1}\cdots n_{j-1}^{i_{j-1}}(n_{j-1}+1)^2}=\\
\eta_{0,1,1}(M_{(i_1,\dots,i_{j-1})})-\zt(2,i_{j-1},\dots,i_1)+
(\eta_{0,2}-\eta_{0,0,2})(M_{(i_1,\dots,i_{j-1})})
\end{multline*}
again using Theorem \ref{nome}.
\end{proof}
It follows that
\begin{multline*}
(\eta_{0,2}-\eta_{0,0,2})(e_k)=1-\zt(k+1)+(\eta_{0,2}-\eta_{0,0,2})(e_{k-1})\\
=\cdots=(k+1)-\zt(k+1)-\zt(k)-\cdots-\zt(2).
\end{multline*}
Similarly,
\[
(\eta_{0,2}-\eta_{0,0,2})(p_k)=\eta_{0,1,1}(p_k)-(\eta_{0,2}-\eta_{0,0,2})(p_{k-1}),
\]
which together with equation (\ref{omp}) implies
\[
(\eta_{0,2}-\eta_{0,0,2})(p_k)=\sum_{j=0}^{k-3}(-1)^j(j+1)\zt(k-j)
+(-1)^kk\zt(2)+(-1)^{k+1}(k+1) .
\]
In view of Corollary \ref{eta02pe}, the preceding equations imply
\[
\eta_{0,0,2}(e_k)=\sum_{j=2}^{k+2}\zt(j)-(k+1)
\]
and
\[
\eta_{0,0,2}(p_k)=\zt(2,k)+\sum_{j=0}^{k-3}(-1)^{j+1}(j+1)\zt(k-j)-(-1)^kk\zt(2)
+(-1)^k(k+1) .
\]
\subsection{General length}
For general length we have Theorem \ref{spiess}, which can be thought of as
generalizing $\eta_{0,1,1}(e_k)=1$.
\begin{proof6}
We use induction on $k$.  For the base case $k=0$ we must show
\[
\frac1{(q-1)!}\cdot\frac1{q-1}=\sum_{n=0}^\infty\frac1{(n+1)(n+2)\cdots (n+q)}
=\sum_{n=1}^\infty\frac1{n(n+1)\cdots (n+q-1)} .
\]
By \cite[Theorem 2]{HM} the right-hand side is
\begin{multline*}
H(\underbrace{1,\dots,1}_q)
=\frac1{(q-1)!}\sum_{i=0}^{q-2}\binom{q-2}{i}\frac{(-1)^i}{i+1}\\
=\frac1{(q-1)!}\cdot\frac1{q-1}\sum_{i=0}^{q-2}(-1)^i\binom{q-1}{i+1}
=\frac1{(q-1)!}\cdot\frac1{q-1} .
\end{multline*}
Now assume inductively that
\[
\eta_{0,\underbrace{\scriptstyle{1,\dots,1}}_q}(e_k)=
\frac1{(q-1)!}\cdot\frac1{(q-1)^{k+1}} .
\]
Then using partial fractions followed by telescoping we have
\begin{multline*}
\eta_{0,\underbrace{\scriptstyle{1,\dots,1}}_q}(e_{k+1})=
\sum_{1\le n_1<\dots<n_{k+2}}\frac1{n_1\cdots n_{k+1}n_{k+2}(n_{k+2}+1)\cdots
(n_{k+2}+q-1)}\\
=\frac1{q-1}\sum_{1\le n_1<\dots<n_{k+2}}\bigg[\frac1{n_1\cdots 
n_{k+1}n_{k+2}(n_{k+2}+1)\cdots (n_{k+2}+q-2)}\\
-\frac1{n_1\cdots n_{k+1}(n_{k+2}+1)\cdots (n_{k+2}+q-1)}\bigg]\\
=\frac1{q-1}\sum_{1\le n_1<\dots<n_{k+1}}\frac1{n_1\cdots n_{k+1}(n_{k+1}+1)
\cdots (n_{k+1}+q-1)}\\
=\frac1{q-1}\eta_{0,\underbrace{\scriptstyle{1,\dots,1}}_q}(e_k)
=\frac1{(q-1)!}\cdot\frac1{(q-1)^{k+2}} ,
\end{multline*}
where we used the induction hypothesis in the last step.
\qed
\end{proof6}
\par\noindent
\begin{rem} J. Spie\ss\ \cite[Theorem 15]{S} provides a formula
for the partial sum
\[
\sum_{n=0}^m\frac{P_k(H_n,H_n^{(2)},\dots,H_n^{(k)})}{(n+1)(n+2)\cdots (n+q)}
\]
which in the limit $m\to\infty$ gives Theorem \ref{spiess}.
\end{rem}
From Theorem \ref{spiess} we can obtain the following result, which generalizes 
the case $k=0$ of Theorems \ref{qpnn1} and \ref{eta111}.
\begin{cor}
For positive integers $k,q$ with $q>1$,
\[
\sum_{n=1}^\infty\frac{P_k(H_n,H_n^{(2)},\dots,H_n^{(k)})}{n(n+1)\cdots (n+q-1)}
=\frac1{(q-1)!}\left[\zt(k+1)-\sum_{j=1}^{q-2}\frac1{j^{k+1}}\right].
\]
\end{cor}
\begin{proof}
It suffices to show that
\begin{equation}
\label{etind}
\eta_{\underbrace{\scriptstyle{1,\dots,1}}_q}(e_k)=
\frac1{(q-1)!}\left[\zt(k+1)-\sum_{j=1}^{q-2}\frac1{j^{k+1}}\right].
\end{equation}
We prove this by induction on $q$, the base case $q=2$ being
the case $l=0$ of Theorem \ref{qpnn1}.  Assume inductively that
equation (\ref{etind}) holds.  By Proposition \ref{diff},
\[
\eta_{\underbrace{\scriptstyle{1,\dots,1}}_{q+1}}(e_k)=
\frac1{q}\left[\eta_{\underbrace{\scriptstyle{1,\dots,1}}_q}(e_k)-
\eta_{0,\underbrace{\scriptstyle{1,\dots,1}}_q}(e_k)\right] .
\]
Using the induction hypothesis (\ref{etind}) and Theorem \ref{spiess},
this can be written
\[
\frac1{q!}\left[\zt(k+1)-\sum_{j=1}^{q-2}\frac1{j^{k+1}}
-\frac1{(q-1)^{k+1}}\right]=\frac1{q!}\left[\zt(k+1)-
\sum_{j=1}^{q-1}\frac1{j^{k+1}}\right] .
\]
\end{proof}


\begin{thebibliography}{99}
\bibitem{BB}
D. Borwein and J. M. Borwein, On an intriguing integral and some
series related to $\zt(4)$, \emph{Proc. Amer. Math. Soc.} {\bf 123}
(1995), 1191-1198.
\bibitem{BC}
J. M. Borwein and O-Y. Chan, Duality in tails of multiple zeta values,
\emph{Int. J. Number Theory} {\bf 6} (2010), 501-514.
\bibitem{ChCh}
X. Chen and W. Chu, The Gauss ${}_2F_1(1)$-summation theorem and harmonic
number identities, \emph{Integral Transforms Spec. Funct.} {\bf 20} (2009), 
925-935.
\bibitem{C}
J. Choi, Summation formulas involving binomial coefficients, harmonic
numbers, and generalized harmonic numbers, \emph{Abst. Appl. Anal.}
{\bf 2014}, art. 501906 (10 pp).
\bibitem{Ch}
W. Chu, Hypergeometric series and the Riemann zeta series, \emph{Acta
Arithmetica} {\bf 82} (1997), 103-118.
\bibitem{Cof}
M. W. Coffey, On one-dimensional digramma and polygamma series related
to the evaluation of Feynman diagrams, \emph{J. Comput. Appl. Math.}
{\bf 183} (2005), 84-100.
\bibitem{Con}
D. F. Connon, Euler-Hurwitz series and nonlinear Euler sums,
preprint {\tt arXiv 0803.1304}, 2008.
\bibitem{CoCa}
M. Coppo and B. Candelpergher, The Arakawa-Kaneko zeta function,
\emph{Ramanujan J.} {\bf 22} (2010), 153-162.
\bibitem{E}
L. Euler, Meditationes circa singulare serierum genus, 
\emph{Novi Comm. Acad. Sci. Petropol.} {\bf 20} (1776), 140-185;
reprinted in \emph{Opera Omnia}, Ser. I, Vol 16(2), B. G. Teubner,
Leipzig, 1935, pp. 104-116.
\bibitem{FS}
P. Flajolet and B. Salvy, Euler sums and contour integral representations,
\emph{Experiment. Math.} {\bf 7} (1998), 15-35.
\bibitem{G}
A. Granville, A decomposition of Riemann's zeta-function, in
\emph{Analytic Number Theory}, Y. Motohashi (ed.), Cambridge University
Press, New York, 1997, 95-100.
\bibitem{H1}
M. E. Hoffman, Multiple harmonic series, \emph{Pacific J. Math.} 
{\bf 152} (1992), 275-290.
\bibitem{H2}
M. E. Hoffman, The algebra of multiple harmonic series, \emph{J. Algebra}
{\bf 194} (1997), 477-495.
\bibitem{H3}
M. E. Hoffman, A character on the quasi-symmetric functions coming from
multiple zeta values, \emph{Electron. J. Combin.} {\bf 15(1)} (2008),
res. paper 97 (21 pp).
\bibitem{H4}
M. E. Hoffman, On multiple zeta values of even arguments, preprint
{\tt arXiv 1205.7051[NT]}, 2012.
\bibitem{HM}
M. E. Hoffman and C. Moen, Sums of generalized harmonic series, \emph{Integers}
{\bf 14} (2014), art. A46 (11 pp).
\bibitem{HO}
M. E. Hoffman and Y. Ohno, Relations of multiple zeta values and their
algebraic expression, \emph{J. Algebra} {\bf 262} (2003), 332-347.
\bibitem{M}
I. G. MacDonald, \emph{Symmetric Functions and Hall Polynomials},
2nd ed., Clarendon Press, Oxford, 1995.
\bibitem{Me}
I. Mez\H o, Nonlinear Euler sums, \emph{Pacific J. Math.} {\bf 272} (2014),
201-226.
\bibitem{So}
A. Sofo, Harmonic number sums in higher powers, \emph{J. Math. Anal.}
{\bf 2} (2011), 15-22.
\bibitem{S}
J. Spie\ss, Some identities involving harmonic numbers,
\emph{Math. Comp.} {\bf 55} (1990), 839-863.
\bibitem{Z}
D. Zagier, Values of zeta functions and their applications, in
\emph{First European Congress of Mathematics, Vol. II (Paris, 1992)}
Birkh\"auser, Basel, 1994, pp. 497-512.
\bibitem{Zh}
D-Y. Zheng, Further summation formulae related to generalized harmonic 
numbers, \emph{J. Math. Anal. Appl.} {\bf 335} (2007), 692-706.
\end{thebibliography}
\end{document}